\documentclass[11pt,a4paper,twoside]{amsart}
\usepackage{amssymb,amsmath,amsthm}
\theoremstyle{plain}
\newtheorem{theorem}{Theorem}[section]

\newtheorem{lemma}[theorem]{Lemma}
\newtheorem{corollary}[theorem]{Corollary}

\theoremstyle{remark}
\newtheorem{remark}[theorem]{Remark}

\usepackage{hyperref}
\numberwithin{equation}{section}

\newcommand{\R}{\mathbb{R}}

\newcommand{\F}{\mathcal{F}}

\renewcommand{\Im}{\operatorname{Im}}
\renewcommand{\Re}{\operatorname{Re}}
\newcommand{\I}{\infty}
\newcommand{\abs}[1]{\left\lvert #1\right\rvert}
\newcommand{\norm}[1]{\left\lVert #1\right\rVert}
\newcommand{\Lebn}[2]{\left\lVert #1 \right\rVert_{L^{#2}}}

\newcommand{\Jbr}[1]{\left\langle #1 \right\rangle}

\def\({\left(}
\def\){\right)}
\def\<{\left\langle}
\def\>{\right\rangle}
\def\le{\leqslant}
\def\ge{\geqslant}

\def\d{{\partial}}
\def\l{\lambda}
\newcommand{\al}{\alpha}

\newcommand{\eps}{\varepsilon}

\begin{document}
\title[Energy solution to 2D-SP system]{
Energy solution to
Schr\"odinger-Poisson system in the two-dimensional whole space}
\author[S. Masaki]{Satoshi Masaki}
\address{Division of Mathematics\\
Graduate School of Information Sciences\\
Tohoku University\\
Sendai 980-8579, Japan}
\email{masaki@ims.is.tohoku.ac.jp}
\begin{abstract}
We consider the Cauchy problem of 
the two-dimensional Schr\"odinger-Poisson system in the energy class.
Though the Newtonian potential diverges at the spatial infinity in the logarithmic order,
global well-posedness is proven in both defocusing and focusing cases.
The key is a decomposition of the nonlinearity into a sum of 
the linear logarithmic potential and a good remainder,
which enables us to apply the perturbation method.
Our argument can be adapted to the one-dimensional problem.
\end{abstract}
\maketitle

\section{Introduction}
This paper is devoted to the study of the Sch\"odinger-Poisson system
\begin{equation}\label{eq:oSP}
	\left\{
	\begin{aligned}
	&i\d_t u + \frac12 \Delta u = \l P u, \quad (t,x) \in \R^{1+2}, \\
	&-\Delta P =  |u|^2, \\
	&u(0,x) = u_0(x),
	\end{aligned}
	\right.
\end{equation}
where $\l$ is a real constant.
We suppose $P$ is the Newtonian potential
\begin{equation}\label{eq:NP}
	P = -\frac{1}{2\pi} (\log |x| * |u|^2)
\end{equation}
where $*$ denotes the convolution.
For a suitable $u$,
this is the unique strong solution of $-\Delta P=|u|^2$ under the condition
\[
	|\nabla P| \to 0 \text{ as }|x|\to \I, \quad
	\nabla P \in L^\I(\R^2), \quad
	P(0)=\int_{\R^2} (\log|y|)|u(y)|^2 dy
\]
(see \cite{Ma2DSP}).
When the dimensions are larger than two, the Schr\"odinger-Poisson
system is a special case of the Hartree equation and
one of the typical example of the nonlinear Schr\"odinger equation
with a nonlocal nonlinearity, and there is large amount of literature
(see \cite{CazBook} and references therein).
%For the one dimensional case, we refer the reader to \cite{DR-JMP,StrSIAMMA,StiMMMAS}.
On the other hand, the two-dimensional case is less studied.
In \cite{AN-MMAS, ZhSIAM}, \eqref{eq:oSP} is considered
with some restrictive assumptions such as a neutrality condition
%(see introduction of \cite{Ma2DSP}).
%These assumptions  are made for saying 
which confirms that the Newtonian potential \eqref{eq:NP}
does not diverge at the spatial infinity and in particular belongs to $L^2$ space.
The Poisson equation is sometimes posed with a background (or doping profile):
\[
	-\Delta P = |u|^2 - b,
\]
where $b$ is a given positive function.
Then, the neutrality condition is $\int|u|^2-b dx=0$ or equivalently $\F(|u|^2-b)(0)=0$.
When we consider the problem in dimensions less than three, this condition is useful to control $P$.
Notice that this condition excludes all nontrivial solutions when $b\equiv0$,
and that we need to remove this condition for the study of \eqref{eq:oSP}.
In \cite{Ma2DSP}, the above assumptions are removed and
the existence of a unique \textit{local} solution is proven
for data in the usual Sobolev space $H^s(\R^2)$ $(s>2)$
despite the fact that the nonlinear potential diverges at the spatial infinity.
% A key of the argument was to introduce 
% a new formula of the solution to the Poisson equation
Since \eqref{eq:NP} is not necessarily defined for $u\in H^s$ ($s>2$)
we introduced a new formula
\[
	P = -\frac{1}{2\pi} \int_{\R^2} \( \log \frac{|x-y|}{|y|}\) |u(y)|^2 dy
\]
which makes sense merely if $|u|^2\in L^p(\R^2)$ ($p\in (1,2)$).
%Another point was a reduction of \eqref{eq:oSP} to a quantum hydrodynamical system.
We underline that the local solutions given there do not have finite energy
(the energy is given in \eqref{eq:energy} below).
Our aim in this paper is to prove that there exists a time-global solution
if initial data has finite energy.

For our analysis, the following reduction is crucial:
We guess that the Newtonian potential \eqref{eq:NP} may behave like
$-\frac{1}{2\pi}\Lebn{u}2^2 \log|x| $ at the spatial infinity,
which will be the bad part of the nonlinearity,
and decompose the nonlinearity as
\[
	\l P u
	=-\frac{\l}{2\pi}\Lebn{u}2^2 (\log\Jbr{x} )u
	-\frac{\l}{2\pi} u\int_{\R^2} \( \log \frac{|x-y|}{\Jbr{x}}\) |u(y)|^2 dy,
\]
where $\Jbr{x}=(1+|x|^2)^{1/2}$.
We then 
%subtract $-\frac{\l}{2\pi}\Lebn{u}2^2 (\log\Jbr{x} )u$ from the 
%both side of \eqref{eq:oSP} to 
obtain
\[
	i\d_t u + \frac12 \Delta u + \frac{\l}{2\pi}\Lebn{u}2^2 (\log\Jbr{x} )u
	= -\frac{\l}{2\pi} u\int_{\R^2} \( \log \frac{|x-y|}{\Jbr{x}}\) |u(y)|^2 dy.
\]
It will turn out that the bad part of $P$ is correctly extracted from the original
nonlinearity and 
therefore the behavior of the ``new nonlinearity'' becomes  better.
Notice that one can also expect that $\Lebn{u}2$ is conserved because $\l$ is a real number.
Hence, putting 
\[
	m:= -\frac{\l}{2\pi}\Lebn{u_0}2^2,
\]
we reach to the equation
\begin{equation}\label{eq:SP}
	\left\{
	\begin{aligned}
	&i\d_t u + \(\frac12 \Delta  -m \log\Jbr{x} \)u
	= -\frac{\l}{2\pi} u\int_{\R^2} \( \log \frac{|x-y|}{\Jbr{x}}\) |u(y)|^2 dy, \\
	&u(0,x) = u_0(x).
	\end{aligned}
	\right.
\end{equation}
Notice that $-m \log\Jbr{x}$ is now completely independent of $u$
and that it therefore can be regarded as a linear potential.
In what follows, we work with this equation.
Observe that
if there exists a solution to \eqref{eq:SP} conserving $\Lebn{u}2$,
then it is also a solution of \eqref{eq:oSP}.

%Our analysis is based on the perturbation method.
Now, the linear part of the equation
is not $i\d_t + (1/2)\Delta$ but $i\d_t + (1/2)\Delta - m \log\Jbr{x}$.
Thus, a natural choice of the function space on which we shall work is
not the Sobolev space $H^1(\R^2)$ any more, but the following one:
\begin{equation}\label{eq:Esp}
\begin{aligned}
	&\mathcal{H}:= \{ u \in H^1(\R^2); \sqrt{\log\Jbr{x}} u \in L^2 \}, \\
	&\norm{u}_{\mathcal{H}} := \norm{u}_{H^1(\R^2)} +
	\norm{\sqrt{\log\Jbr{\cdot}} u}_{L^2(\R^2)}.
\end{aligned}
\end{equation}
If $m>0$, that is, if $\l<0$, then
the above space coincides with the form domain of 
the positive operator $-\frac{1}2\Delta + m \log\Jbr{x}$.
%We take initial data from $\mathcal{H}$.
%\subsection{Main result}
Our main result is the following:
\begin{theorem}\label{thm:main}
The problem \eqref{eq:SP} is globally well-posed in $\mathcal{H}$.
Moreover, the solution conserves $\Lebn{u(t)}2$ and the energy
\begin{equation}\label{eq:energy}
	E(t) = \frac12 \Lebn{\nabla u(t)}2^2
	-\frac{\l}{4\pi} \int_{\R^2} (\log|x-y|)|u(t,x)|^2|u(t,y)|^2 dx dy.
\end{equation}
\end{theorem}
\begin{corollary}
The Problem \eqref{eq:oSP} is globally well-posed in $\mathcal{H}$.
\end{corollary}
\begin{remark}
Let $u\in C(\R;\mathcal{H})$
be a solution of \eqref{eq:SP} (and of \eqref{eq:oSP}) given in
Theorem \ref{thm:main}.
Then, $v:= u \exp (-i\frac{\l}{2\pi}\int_0^t \|\sqrt{\log|\cdot|}u(s)\|_{L^2}^2ds)$ solves
\begin{equation}\label{eq:SPp}
	\left\{
	\begin{aligned}
	&i\d_t v + \frac12 \Delta v
	= -\frac{\l}{2\pi} v\int_{\R^2} \( \log \frac{|x-y|}{|y|}\) |v(y)|^2 dy, \\
	&v(0,x) = u_0(x).
	\end{aligned}
	\right.
\end{equation}
Notice that the nonlinearity of \eqref{eq:SPp} makes sense without
the momentum condition $\sqrt{\log|\cdot|}v \in L^2$. 
This observation explains why existence of a time-local solution can be proven
by assuming only $u_0 \in H^s(\R^2)$ ($s>1$) in \cite{Ma2DSP}.
\end{remark}
\subsection{Consequent results}
Our argument is also applicable to \eqref{eq:oSP}
involving a power type nonlinearity:
\begin{equation}\label{eq:oSPp}
	\left\{
	\begin{aligned}
	&i\d_t u + \frac12 \Delta u = \l P u + \eta |u|^{p-1}u, \quad (t,x) \in \R^{1+2}, \\
	&-\Delta P =  |u|^2, \\
	&u(0,x) = u_0(x),
	\end{aligned}
	\right.
\end{equation}
where $\eta$ is a real number and $p\ge2$.
\begin{theorem}\label{thm:withpower}
The problem \eqref{eq:oSPp} is globally well-posed in $\mathcal{H}$ if either 
one of the following conditions is satisfied:
\begin{enumerate}
\item $\eta\ge0$, $\l \in \R$ and $p\ge 2$; 
\item $\eta<0$, $\l \in \R$, and $2\le p<3$;
\item $\eta<0$, $\l>0 $, $p=3$, and $\norm{u_0}_{\mathcal{H}}$ is small;
\item $\eta<0$, $\l<0$, $p\ge3$, and $\norm{u_0}_{\mathcal{H}}$ is small.
\end{enumerate}
Moreover, the solution conserves $\Lebn{u(t)}2$ and the energy
\begin{equation}\label{eq:Ep}
\begin{aligned}
	E_p(t):={}&\frac12 \Lebn{\nabla u(t)}2^2
	-\frac{\l}{4\pi} \int_{\R^2} (\log|x-y|)|u(t,x)|^2|u(t,y)|^2 dx dy \\
	&{} + \frac{\eta}{p+1}\Lebn{u(t)}{p+1}^{p+1}.
\end{aligned}
\end{equation}
\end{theorem}
The proof is done with a straight-forward modification (see Section \ref{sec:power}).
The case where $p=3$ is known as the $L^2$-critical case.
Since the $\mathcal{H}$-norm contains derivative, 
it seems difficult to treat the case  $1< p <2$.
Nevertheless, we can show global well-posed in a slightly smaller function space
$\mathcal{H}^{1,2}:= \{ u \in H^1(\R^2); u\log\Jbr{x} \in L^2 \}$.

\begin{theorem}\label{thm:lowpower}
Suppose $1<p<2$. For $\eta,\l \in \R$
The problem \eqref{eq:oSPp} is globally well-posed in the space
$\mathcal{H}^{1,2}$.
Moreover, the solution conserves $\Lebn{u(t)}2$ and the energy
$E_p(t)$ given in \eqref{eq:Ep}.
\end{theorem}
\smallbreak

We can also handle the one-dimensional problem
\begin{equation}\label{eq:SP1}
	\left\{
	\begin{aligned}
	&i\d_t u + \frac12 \d_{xx} u = -\frac{\l}2 (|x|*|u|^2)u + \eta |u|^{p-1}u
, \quad (t,x) \in \R^{1+1}, \\
	&u(0,x) = u_0(x),
	\end{aligned}
	\right.
\end{equation}
where $\l,\eta \in \R$ and $p\ge2$.
The one dimensional problem was studied in \cite{DR-JMP,StrSIAMMA,StiMMMAS}.
The global well-posedness of \eqref{eq:SP1} was shown 
in the space $\{f \in H^1(\R); |x|f\in L^2(\R)\}$ in \cite{StiMMMAS},
and in the space $\{f \in H^1(\R); \sqrt{|x|}f\in L^2(\R)\}$
with a presence of background in \cite{DR-JMP}, provided
$\l>0$ and data is small relative to the background.
We can prove the global well-posedness result of \eqref{eq:SP1}
including these results.
\begin{theorem}\label{thm:1D}
The problem \eqref{eq:SP1} is globally well-posed in 
$\{f \in H^1(\R); \sqrt{|x|}f\in L^2(\R)\}$ if $\l\in \R$ and either
one of the following conditions is satisfied:
\begin{enumerate}
\item $\eta\ge0$, $\l\in\R$, and $p\ge 2$;
\item $\eta<0$, $\l\in\R$, and $2\le p<5$;
\item $\eta<0$, $\l >0$, $p=5$, and $\norm{u_0}_{H^1}+\|{\sqrt{|\cdot|}u_0}\|_{L^2}$ is small;
\item $\eta<0$, $\l <0$, $p\ge5$, and $\norm{u_0}_{H^1}+\|{\sqrt{|\cdot|}u_0}\|_{L^2}$ is small.
\end{enumerate}
The solution conserves $\Lebn{u}2$ and the energy
\begin{equation}\label{eq:1DE}
	\widetilde{E}(t):=\frac12\norm{\d_x u}_{L^2(\R)}^2 - \frac{\l}2 \iint_{\R^2} |x-y| |u(x)|^2|u(y)|^2 dxdy
	+\frac{\eta}{p+1}\norm{u}_{L^{p+1}(\R)}^{p+1}.
\end{equation}
\end{theorem}
The one-dimensional version of Theorem \ref{thm:lowpower} is as follows,
which reproduce the same result in 
\cite[Theorem 2.1]{StiMMMAS} when $\eta<0$ and $\l>0$.
\begin{theorem}\label{thm:1Dlow}
Suppose $1<p<2$. For $\eta,\l \in \R$
The problem \eqref{eq:SP1} is globally well-posed in the space
$\Sigma:= \{ u \in H^1(\R^2); |x|u \in L^2 \}$.
% with norm
% \[
% 	\norm{u}_{\mathcal{H}} := \norm{u}_{H^1(\R^2)} +
% 	\norm{\log\Jbr{\cdot} u}_{L^2(\R^2)}.
% \]
Moreover, the solution conserves $\Lebn{u(t)}2$ and the energy
$\widetilde{E}(t)$ given in \eqref{eq:1DE}.
% \[
% 	\frac12\norm{\d_x u}_{L^2(\R)}^2 - \frac{\l}2 \iint_{\R^2} |x-y| |u(x)|^2|u(y)|^2 dxdy
% 	+\frac{\eta}{p+1}\norm{u}_{L^{p+1}(\R)}^{p+1}.
% \]
\end{theorem}
As in the two dimensional case, the key is a ``reduction'' of \eqref{eq:SP1} to
\begin{equation*}%\label{eq:SP1}
	\left\{
	\begin{aligned}
	&i\d_t u + \frac12 \d_{xx} u + \frac{\l \Lebn{u_0}2^2}2 |x| u  =
	-\frac{\l}2 u\int_\R {(|x-y|-|x|)}|u(y)|^2 dy + \eta |u|^{p-1}u,  \\
	&u(0,x) = u_0(x).
	\end{aligned}
	\right.
\end{equation*}
\smallbreak

We briefly mention about other related works.
Oh considered in \cite{OhJDE} the Cauchy problem of 
the nonlinear Schr\"odinger equation with general potential and
$L^2$-subcritical power-type nonlinearity, and
proved global well-posedness in the form domain of $-\frac12 \Delta + V$,
provided the potential $V\ge0$ satisfies $\d^\al V \in L^\I$ for $|\al|\ge2$
(see also \cite{CazBook}).
In particular, the case where the potential $V$ is a quadratic polynomial
is extensively studied.
In this case, we have several special properties such as explicit representations
of linear solutions, called Mehler's formula, and/or of the Heisenberg observables.
We refer the reader to \cite{CaAHP,CaSIAMMA,CaDCDS,KVZ-CPDE,ZxFM}
for $H^1$-subcritical and $H^1$-critical power-type nonlinearity 
and to \cite{CMS-SIAM} for $H^1$-subcritical Hartree type nonlinearity.
In \cite{Stu2DSN}, the ground states of \eqref{eq:oSP} is treated.
\smallbreak

The rest of the paper is organized as follows:
We collect some basic estimates in Section \ref{sec:pre}, 
and, in Section \ref{sec:proof} we prove Theorem \ref{thm:main}.
Section \ref{sec:power} is devoted to the study of \eqref{eq:oSPp}.

\section{Preliminaries}\label{sec:pre}
\subsection{Strichartz estimate}
We first summarize the properties on 
the operator
\begin{equation}\label{eq:A}
	A: = \frac12 \Delta - m \log \Jbr{x},
\end{equation}
where $m\neq 0$ is a real constant.
For any $m$, $A$ is essentially self-adjoint on $C_0^\I(\R^2)$
(see \cite{RSBook2}).
Since our potential is sub-quadratic, that is,
since $|\d^\alpha \log\Jbr{x}|\to0$ as $|x|\to\I$ for $|\al|=2$
and $\d^\alpha \log\Jbr{x} \in L^\I$ for $|\al| \ge 3$,
the following estimate is established in \cite{YajCMP}:
For any $T>0$,
\[
	\Lebn{e^{itA} \varphi}\I \le C|t|^{-1} \Lebn{\varphi}1
\]
for $t \in [-T,T]$, where $C$ depends on $T$ (see also \cite{FujDMJ}).
Once we know this type of estimate, the Strichartz estimate follows by interpolation.
%(see \cite{KT-AJM} and references therein).
We say that a pair $(q,r)$ is admissible if $2\le r <\I$ and $2/q=\delta(r):=1-2/r$.
\begin{lemma}[Strichartz's estimate]
For any $T>0$, 
the following properties hold:
\begin{itemize}
\item 
Suppose $\varphi \in L^2(\R^2)$.
For any admissible pair $(q,r)$, there exists a constant $C=C(T,q,r)$
such that
\[
	\norm{e^{itA} \varphi}_{L^q((-T,T);L^r)} \le C \Lebn{\varphi}2.
\]
\item Let $I \subset (-T,T)$ be an interval and $t_0 \in \overline{I}$. 
For any admissible pairs $(q,r)$ and $(\gamma,\rho)$,
there exists a constant $C=C(t,q,r,\gamma,\rho)$ such that
\[
	\norm{\int_{t_0}^t e^{i(t-s)A} F(s) ds}_{L^q(I;L^r)}
	\le C \norm{F}_{L^{\gamma^\prime}(I; L^{\rho^\prime})}
\]
for every $F \in L^{\gamma^\prime}(I; L^{\rho^\prime})$.
\end{itemize}
\end{lemma}

\subsection{Some estiamtes}
\begin{lemma}\label{lem:commutatorA}
Let $W$ be an arbitrary weight function such that $\nabla W$, $\Delta W \in L^\I(\R^2)$.
It holds for all $T >0$, admissible pair $(q,r)$,
and $\varphi \in \mathcal{H}$ that
\begin{align*}
	\norm{[\nabla ,e^{itA}] \varphi}_{L^{q}((-T,T);L^{r})}
	\le {}& C |T| \norm{\varphi}_{2}, \\
	\norm{[W ,e^{itA}] \varphi}_{L^{q}((-T,T);L^{r})}
	\le {}& C |T| \norm{(1+\nabla)\varphi}_{2}. \\
\end{align*}
\end{lemma}
\begin{proof}
Since $v=e^{itA}\varphi$ solves $i\d_t v + Av =0$,
an explicit calculation shows
\[
	[\nabla ,e^{itA}] \varphi= - i\int_0^t e^{i(t-s)A}\frac{mx}{1+x^2} e^{isA}\varphi ds
\]
and
\[
	[W ,e^{itA}] \varphi=  i\int_0^t e^{i(t-s)A} 
	\(\nabla W\cdot \nabla + \frac12 \Delta W\) e^{isA}\varphi ds.
\]
The Strichartz estimate therefore gives the desired estimates.
\end{proof}

% \subsection{An estimate}
% We establish an estimate which we use for estimates on the
% nonlinearity of \eqref{eq:SP}.
The following is useful for estimates of the nonlinearity in \eqref{eq:SP}.
\begin{lemma}\label{lem:estK}
Set a function
\[
	K(x,y) = \frac{\log \frac{|x-y|}{\Jbr{x}}}{1+\log\Jbr{y}}
\]
of $x,y \in \R^2$.
For any $p \in [1,\I)$ and $\eps>0$, there exist a function $W(x,y)\ge 0$
with $\norm{W}_{L^\I_y L^p_x} \le \eps$ and a constant $C_0$
such that
\[
	|K(x,y)| \le C_0 + W(x,y)
\]
holds for all $(x,y) \in \R^{2+2}$.
\end{lemma}
\begin{proof}
Take $\eta \in (0,1]$ and set $W(x,y) = |K(x,y)| {\bf 1}_{|x-y|\le \eta}(x,y)$.
If $\eta$ is sufficiently small then
\[
	\Lebn{W(\cdot,y)}p \le \frac{\norm{\log|x|}_{L^p(|x|\le \eta)} + \log\Jbr{|y|+\eta}
	\norm{1}_{L^p(|x|\le \eta)}}{1+ \log\Jbr{y}}\le \eps
\]
since $\log|x|$ belongs to $L^p_{\mathrm{loc}}(\R^2)$ for all $p<\I$.
Moreover,
by (2.12) of \cite{Ma2DSP}, 
\[
	\sup_{|x-y|\ge \eta} K(x,y) \le 1+ \log \frac{\sqrt3}{\eta}
\]
for any $\eta\le1$, which completes the proof.
\end{proof}
\begin{remark}
In 1D case, the corresponding estimate is
\[
	\norm{\frac{|x-y|-|x|}{1+|y|}}_{L^\I_{x,y}(\R^2)}
	\le 1.
\]
\end{remark}

\section{Proof of the theorem}\label{sec:proof}
\subsection{Local well-posedness}
\begin{lemma}\label{lem:LWP}
Let $(q_0,r_0)$ be an admissible pair with $r_0>2$.
For any $u_0 \in \mathcal{H}$, there exist
an existence time $T=T(\norm{u_0}_{\mathcal{H}})$ 
and a unique solution $u \in C((-T,T);\mathcal{H}) \cap
L^{q_0}((-T,T);L^{r_0}) \cap C^1((-T,T);\mathcal{H}^{*})$.
The solution conserves $\Lebn{u(t)}2$ and the energy \eqref{eq:energy}.
Moreover, the map $u_0\mapsto u$ is continuous from $\mathcal{H}$ to
$C((-T,T);\mathcal{H})$.
\end{lemma}
\begin{proof}
We write $L^p((-T,T);X)=L^p_T X$, for short.
Define a Banach space
\[
	\mathcal{H}_{T,M} := \{f \in L^\I((-T,T);\mathcal{H});
	\norm{f}_{\mathcal{H}_{T}} \le M \}
\]
with norm
\begin{align*}
	\norm{f}_{\mathcal{H}_{T}} :={}& 
	\norm{f}_{L^\I_T \mathcal{H}} + \norm{f}_{L^{q_0}_T W^{1,r_0}}
	+ \norm{\sqrt{\log\Jbr{x}} f}_{L^{q_0}_T L^{r_0}}.
\end{align*}
We show that if $r_0>2$ then 
there exist $M=M(\norm{u_0}_{\mathcal{H}})$ and
$T=T(\norm{u_0}_{\mathcal{H}})$ such that
\begin{multline*}
	Q[u](t,x) := (e^{itA} u_0)(x) \\
	+ \frac{i}{2\pi} \(\int_{0}^t e^{i(t-s)A} 
	\(\int_{\R^2} \log\frac{|\cdot-y|}{\Jbr{\cdot}} |u(s,y)|^2 dy\) u(s,\cdot) ds\)(x)
\end{multline*}
becomes a contraction map from $\mathcal{H}_{T,M}$ to itself,
where $A$ is defined in \eqref{eq:A}.
\smallbreak

Set
\[
	K(x,y) = \frac{\log \frac{|x-y|}{\Jbr{x}}}{1+\log\Jbr{y} }.
\]
Then, by Lemma \ref{lem:estK}, 
there exist a nonnegative function $W\in L^\I_y L_x^{r_0^\prime}$ and a
constant $C_0$ such that
\[
	|K(x,y)| \le C_0 + W(x,y).
\]
Recall that $r_0 \in (2,\I)$ and so $r_0^\prime:=r_0/(r_0-1) \in (1,2)$.
We hence see that
\[
	Pu=\iint K(x,y) (1+\log\Jbr{y} )|u(y)|^2 u(x) dy\, dx	
\]
satisfies
\[
	\Lebn{Pu}2 \le C (\Lebn{u}2 + \Lebn{u}{r_0})\Lebn{\sqrt{1+\log\Jbr{x}}u}2^2.
\]
Take $L^1_T$ norm to yield
\begin{equation}\label{eq:esttmp1}
	\norm{Pu}_{L^1_TL^2} \le C
	(T\norm{u}_{L^\I_TL^2} + T^{\frac12+\frac1{r_0}}\norm{u}_{L^{q_0}_TL^{r_0}})
	\norm{\sqrt{1+\log\Jbr{x}}u}_{L^\I_TL^2}^2.
\end{equation}
By the Strichartz estimate, we end up with
\begin{equation}\label{eq:estQ1}
	\norm{Q[u]}_{L^\I_T L^2} + \norm{Q[u]}_{L^{q_0}_T L^{r_0}}
	\le C\Lebn{u_0}2 + 
	C(T+ T^{\frac12+\frac1{r_0}})\norm{u}_{\mathcal{H}_{T}}^3.
\end{equation}

We next estimate $\nabla Q[u]$.
One easily sees that
\begin{align*}
	\nabla Q [u] ={}& e^{itA} \nabla u_0 - i\int_0^t e^{i(t-s)A} \nabla (Pu)(s)ds \\
	&{}+[\nabla, e^{itA} ] u_0 - i\int_0^t [\nabla, e^{i(t-s)A}] (Pu)(s)ds.
\end{align*}
We deduce from Lemma \ref{lem:commutatorA} with $(q,r)=(\I,2)$ that
\[
	\int_0^t \Lebn{[\nabla, e^{i(t-s)A}] (Pu)(s)}2 ds
	\le \int_0^t (t-s) \Lebn{Pu(s)}2 ds
	\le |t| \norm{Pu}_{L^1_TL^2}.
\]
The right hand side is bounded as in \eqref{eq:esttmp1}.
$[\nabla, e^{itA} ] u_0$ is handled similarly.
Mimicking \eqref{eq:esttmp1}, we infer that
\begin{equation}\label{eq:esttmp2}
	\norm{P\nabla u}_{L^1_TL^2}\le 
	C(T\norm{\nabla u}_{L^\I_TL^2} + T^{\frac12+\frac1{r_0}}\norm{\nabla u}_{L^{q_0}_TL^{r_0}})
	\norm{\sqrt{1+\log\Jbr{x}}u}_{L^\I_TL^2}^2
\end{equation}
Now, let us estimate $(\nabla P)u$. It writes
\[
	(\nabla P) (x)u(x) = 
	\(\int_{\R^2}\(\frac{x-y}{|x-y|^2} - \frac{x}{1+x^2}\)|u(y)|^2 dy \)u(x) ,
\]
and so
\begin{align*}
	\Lebn{(\nabla P) u}2
	\le {}& C \Lebn{(|x|^{-1}*|u|^2) + \Jbr{\cdot}^{-1}\Lebn{u}2^2}{\frac{2r_0}{r_0-2}} \Lebn{u}{r_0}\\
	\le {}& C(\Lebn{u}{\frac{2r_0}{r_0-1}}^2+\Lebn{u}2^2) \Lebn{u}{r_0}\\
	\le {}& C(\Lebn{u}2^2+\Lebn{\nabla u}2^2) \Lebn{u}{r_0}
\end{align*}
by the Hardy-Littlewood-Sobolev and the Sobolev inequalities.
We see that
\begin{equation}\label{eq:esttmp3}
	\norm{(\nabla P)u}_{L^1_TL^2} \le
	C(T \norm{u}_{L_T^\I L^2}^2 + T^{\frac12+\frac1{r_0}}\norm{u}_{L_T^{q_0} L^{r_0}}^2) \norm{u}_{L_T^\I L^2}.
\end{equation}
We deduce from the Strichartz estimate that
\begin{equation}\label{eq:estQ2}
	\norm{\nabla Q[u]}_{L^\I_TL^2} + \norm{\nabla Q[u]}_{L^{q_0}_TL^{r_0}}
	\le C\norm{\nabla u_0}_{\mathcal{H}}
	+ C(T+T^{\frac12+\frac1{r_0}})\norm{u}_{\mathcal{H}_{T}}^3.
\end{equation}

Let us proceed to the estimate of $\sqrt{\log \Jbr{x}} Q[u]$.
It holds that
\begin{align*}
	\sqrt{1+\log\Jbr{x}} Q[u] ={}&
	e^{itA} \sqrt{1+\log\Jbr{x}} u_0 - i\int_0^t e^{i(t-s)A} \sqrt{1+\log\Jbr{x}}Pu(s)ds\\
	&{} +R,
\end{align*}
where
\[
	R = [\sqrt{1+\log\Jbr{x}}, e^{itA} ]u_0
	- i\int_0^t [ \sqrt{1+\log\Jbr{x}}, e^{i(t-s)A}]Pu(s)ds.
\]
A use of Lemma \ref{lem:commutatorA} with
$W=\sqrt{1+\log\Jbr{x}}$ yields
\begin{align*}
	\norm{R}_{L^\I_T L^2}+ \norm{R}_{L^{q_0}_T L^{r_0}} 
	&{}\le CT\norm{u_0}_{\mathcal{H}} + CT\norm{(1+\nabla)(Pu)}_{L^1_TL^2}\\
	&{}\le CT\norm{u_0}_{\mathcal{H}} + 
	CT(T + T^{\frac12+\frac1{r_0}}) \norm{u}_{\mathcal{H}_T}^3
\end{align*}
where we have used \eqref{eq:esttmp1}, \eqref{eq:esttmp2}, and \eqref{eq:esttmp3}. 
As in \eqref{eq:esttmp1}, it holds that
\begin{align*}%\label{eq:esttmp4}
	\norm{P(Wu)}_{L^1_TL^2} \le{}& C
	(T\norm{Wu}_{L^\I_TL^2} + T^{\frac12+\frac1{r_0}}\norm{Wu}_{L^{q_0}_TL^{r_0}})
	\norm{Wu}_{L^\I_TL^2}^2\\
	\le {}& C (T+T^{\frac12+\frac1{r_0}})\norm{u}_{\mathcal{H}_T}^3,
\end{align*}
where $W=\sqrt{1+\log\Jbr{x}}$.
We conclude from the Strichartz estimate, \eqref{eq:estQ1}, and \eqref{eq:estQ2} that
\begin{equation*}
	\norm{ Q[u]}_{\mathcal{H}_T} 
	\le C_1\norm{u_0}_{\mathcal{H}}
	+ C_2(T+T^{\frac12+\frac1{r_0}})\norm{u}_{\mathcal{H}_{T}}^3.
\end{equation*}
A similar argument shows
\begin{equation*}
	\norm{ Q[u_1]- Q[u_2]}_{\mathcal{H}_T} 
	\le  C_3(T+T^{\frac12+\frac1{r_0}})(\norm{u_1}_{\mathcal{H}_{T}}
	+\norm{u_2}_{\mathcal{H}_{T}} )^2 \norm{u_1-u_2}_{\mathcal{H}_{T}}.
\end{equation*}
Thus, if we take $M\ge 2C_1\norm{u_0}_{\mathcal{H}}$
then there exists $T=T(M)$ such that $Q$ is a contraction map from $\mathcal{H}_{T,M}$ to itself.

The conservations of $\Lebn{u(t)}2$ is shown by multiplying \eqref{eq:SP}
by $\overline{u}$ and integrating the imaginary part.
To prove the energy conservation, we need a regularizing argument.
Note that \eqref{eq:SP} can be solved also in the space
$\{f \in H^2(\R^2): \log\Jbr{x} f \in L^2 \}$,
which is one of dense subsets of $\mathcal{H}$,
in an essentially same way. We omit details.
\end{proof}

\subsection{Global existence}
We first give a useful blow-up criteria.
\begin{lemma}\label{lem:BA}
Suppose $u_0 \in \mathcal{H}$.
Let $u \in C((-T_{\mathrm{min}},T_{\mathrm{max}});\mathcal{H})$
be a unique maximal solution given by Lemma \ref{lem:LWP}.
If $T_{\mathrm{max}}<\I$ (resp. $T_{\mathrm{min}}<\I$), then
$\Lebn{\nabla u(t)}2 \to \I$ as $t \uparrow T_{\mathrm{max}}$
(resp. $t \downarrow -T_{\mathrm{min}}$).
\end{lemma}
\begin{proof}
We only consider positive time.
Suppose $T_{\mathrm{max}}<\I$.
Then, $\norm{u(t)}_{\mathcal{H}}$ has to diverge as $t \uparrow T_{\mathrm{max}}$.
Otherwise, we can extend the solution beyond $T_{\mathrm{max}}$
by Lemma \ref{lem:LWP}.
Recall that $\Lebn{u(t)}2=\Lebn{u_0}2$.
Since
\begin{align*}
	\frac{d}{dt} \Lebn{\sqrt{\log\Jbr{x}} u(t)}2^2
	&{}= 2\Re \int (\log\Jbr{x}) \d_t u(t) \overline{u(t)} dx\\
	&{}= - \Im \int (\log\Jbr{x}) \Delta u(t) \overline{u(t)} dx \\
	&{}=  \Im \int \frac{x }{1+x^2}\cdot \nabla u(t) \overline{u(t)} dx,
\end{align*}
it holds that
\[
	\Lebn{\sqrt{\log\Jbr{x}} u(t_2)}2^2 \le
	\Lebn{\sqrt{\log\Jbr{x}} u(t_1)}2^2
	+ |t_2-t_1| \norm{\nabla u}_{L^\I((t_1,t_2);L^2)}\Lebn{u_0}2
\]
for all $-T_{\mathrm{min}}< t_1 <t_2 < T_{\mathrm{max}}$.
This implies that 
if we assume
\[
	\limsup_{t\uparrow T_{\mathrm{max}}}\Lebn{\nabla u(t)}2 <\I
\]
then $\norm{u(t)}_{\mathcal{H}}$ never blows up.
We hence obtain the lemma.
\end{proof}
\begin{remark}
As in \cite{GlJMP}, the solution breaks down with concentration at a point
if $\|\sqrt{\log\Jbr{x}}u(t)\|_{L^2}=0$. 
However, this does not occur when $\|\nabla u(t)\|$ is bounded above.
Indeed, since
\[
	\norm{u}_{L^2(|x|<r)} \le \norm{1}_{L^4(|x|<r)}\norm{u}_{L^4}
	\le Cr^{\frac12}\Lebn{\nabla u}2^{\frac12}
\]
for any $r>0$ and since
\[
	\norm{u}_{L^2(|x|<r)} = \Lebn{u_0}2 - \norm{u}_{L^2(|x|\ge r)}
	\ge \Lebn{u_0}2 - \frac{\|\sqrt{\log\Jbr{x}}u\|_{L^2}}{(\log\Jbr{r})^{1/2}},
\]
by letting $r=\|\sqrt{\log\Jbr{x}}u\|_{L^2}$, we obtain
\[
	\|\sqrt{\log\Jbr{x}}u\|_{L^2}^{-\frac12}
	\le C\(\frac{\|\sqrt{\log\Jbr{x}}u\|_{L^2}}{\log\Jbr{\|\sqrt{\log\Jbr{x}}u\|_{L^2}}}\)^{\frac12} + C\Lebn{\nabla u}2^{\frac12},
\]
which implies $\|\sqrt{\log\Jbr{x}}u\|_{L^2}$ is strictly positive if
$\Lebn{\nabla u}2<\I$.
\end{remark}

\begin{proof}[Proof of Theorem \ref{thm:main}]
Let us establish a priori estimate of $\norm{\nabla u(t)}_{L^2}$.
\smallbreak

We first consider the case $\l<0$.
Since $\log |x| \ge0$ for $|x|\ge1$,
\begin{align*}
	&-\frac{\l}{4\pi}\iint_{\R^{2+2}} \log |x-y| |u(x)|^2 |u(y)|^2 dxdy\\
	&{}	\ge - \frac{|\l|}{4\pi}\iint_{|x-y|<1} \abs{\log|x-y|}|u(x)|^2 |u(y)|^2 dxdy \\
	&{}\ge -  \frac{|\l|}{4\pi}\norm{\log |x|}_{L^2(|x|\le 1)}^2 \Lebn{u}4^2  \Lebn{u}2^2
\end{align*}
By the $L^2$-conservation and the Sobolev embedding,
we have
\begin{equation}\label{eq:l-case}
	\norm{\nabla u(t)}_{L^2}^2 \le 2E_0 + C \norm{\nabla u(t)}_{L^2}.
\end{equation}
Therefore, there exists a constant $M$ independent of $t$
such that $\norm{\nabla u(t)}_{L^2}\le M$.
\smallbreak

We now suppose $\l>0$.
By Lemma \ref{lem:estK}, for any $\eps>0$ there exists a constant $C_0$ such that
the following estimate holds:
\begin{align*}
	&\frac{\l}{4\pi}\iint_{\R^{2+2}} \log |x-y| |u(x)|^2 |u(y)|^2 dxdy\\
	&{} \le \frac{\l}{4\pi}\iint_{\R^{2+2}} \log \frac{|x-y|}{\Jbr{x}} |u(x)|^2 |u(y)|^2 dxdy
	+ \frac{\l}{4\pi}\Lebn{u_0}2^2 \Lebn{\sqrt{\log\Jbr{x}}u}2^2 \\
	&{}\le \frac{\l}{4\pi}(C_0 \Lebn{u_0}2^2 + \eps \Lebn{u}4^2)
	\Lebn{\sqrt{1+\log\Jbr{x}}u}2^2+ \frac{\l}{4\pi}\Lebn{u_0}2^2 \Lebn{\sqrt{\log\Jbr{x}}u}2^2\\
	&{} \le \frac{\l C_0}{4\pi} \Lebn{u_0}2^4 + \frac{\l(C_0+1)}{4\pi}\Lebn{u_0}2^2\Lebn{\sqrt{\log\Jbr{x}}u}2^2 + C \eps\Lebn{u_0}2^3 \Lebn{\nabla u}2\\
	& \quad + C\eps \Lebn{u_0}2 \Lebn{\nabla u}2 \Lebn{\sqrt{\log\Jbr{x}}u}2^2\\
	&{} \le C_1 + C_2(\eps+|t|) \sup_{s\in[0,t]}\Lebn{\nabla u(s)}2 +
	C_3 \eps |t| \sup_{s\in[0,t]}\Lebn{\nabla u(s)}2^2,  
\end{align*}
where $C_i$ ($i=1,2,3$) depends only on $\l$, $C_0$, $\norm{u_0}_{\mathcal{H}}$, and $\eps$.
Fix $T>0$. Taking $\eps< (8C_3T)^{-1}$, we deduce from the conservation of
$E(t)$ that
\begin{equation}\label{eq:l+case}
	\(\sup_{s\in[0,t]}\Lebn{\nabla u(s)}2\)^2 \le 4E(0) + 4C_1 + 4C_2(\eps+2T) \sup_{s\in[0,t]}\Lebn{\nabla u(s)}2
\end{equation}
for $0 \le t \le 2T$. This implies that
\[
	\sup_{t\in [0,2T] } \Lebn{\nabla u(t)}2 \le C(\norm{u_0}_{\mathcal{H}},T)<\I.
\]
Since $T$ is arbitrary, we obtain the global existence.
\end{proof}

\section{Remarks on the problem with power nonlinearity}\label{sec:power}
We give a rough sketch of the proofs of Theorem \ref{thm:withpower} and \ref{thm:lowpower} in this section.

\begin{proof}[Proof of Theorem \ref{thm:withpower}]
The local well-posedness part
holds if $p\ge2$ as in the proof of Lemma \ref{lem:LWP}.
The restriction $p\ge2$ is required when we estimate
\begin{multline*}
	|\nabla(|u_1|^{p-1} u_1 - |u_2|^{p-1} u_2)|
	\le C_p(|u_1|^{p-2}+ |u_2|^{p-2})(|\nabla u_1| + |\nabla u_2|)|u_1-u_2|\\
	+ C_p(|u_1|^{p-1}+ |u_2|^{p-1})|\nabla(u_1-u_2)|.
\end{multline*}
By exactly the same argument as in Lemma \ref{lem:BA},
the problem of global existence boils down to obtaining an a priori 
bound of $\Lebn{\nabla u(t)}2$.
Recall that the conserved energy is
\begin{align*}
	E_p(t):={}&\frac12 \Lebn{\nabla u(t)}2^2
	-\frac{\l}{4\pi} \int_{\R^2} (\log|x-y|)|u(t,x)|^2|u(t,y)|^2 dx dy \\
	 &{}+ \frac{\eta}{p+1}\Lebn{u(t)}{p+1}^{p+1}.
\end{align*}
\subsubsection*{The case $\eta>0$}
We have
\[
	\Lebn{\nabla u(t)}2^2
	\le E_p(t) + \frac{\l}{4\pi}\int_{\R^2} (\log|x-y|)|u(t,x)|^2|u(t,y)|^2 dx dy.
\]
By the same argument as in the case $\eta=0$, we prove global existence.
\subsubsection*{The case $\eta<0$ and $\l<0$}
Since
\[
	\frac{|\eta|}{p+1}\Lebn{u(t)}{p+1}^{p+1}
	\le C_{\eta,p} \Lebn{u_0}2^2 \Lebn{\nabla u(t)}2^{p-1},
\]
we obtain
\[
	\norm{\nabla u(t)}_{L^2}^2 \le 2E_0 + C \norm{\nabla u(t)}_{L^2}
	+ 2C_{\eta,p}\Lebn{u_0}2^2 \Lebn{\nabla u(t)}2^{p-1}
\]
as in \eqref{eq:l-case}.
Uniform bound of $\Lebn{\nabla u(t)}2$ is then obtained either
the case $p<3$ or the case $p\ge3$ and $\Lebn{u_0}2$ is small.
\subsubsection*{The case $\eta<0$ and $\l>0$}
As in \eqref{eq:l+case}, for any $T>0$,
there exist $\eps$, $C_1$, and $C_2$ such that
\begin{align*}
	\(\sup_{s\in[0,t]}\Lebn{\nabla u(s)}2\)^2 \le{}& 4E(0) + 4C_1 + 4C_2(\eps+2T) \sup_{s\in[0,t]}\Lebn{\nabla u(s)}2\\
	&{} + 4C_{\eta,p}\Lebn{u_0}2^2 \(\sup_{s\in[0,t]}\Lebn{\nabla u(s)}2\)^{p-1}.
\end{align*}
for $t\le 2T$.
Therefore, if $p<3$ or if $p=3$ and $\Lebn{u_0}2$ is small,
we obtain
\[
	\sup_{t\in [0,2T] } \Lebn{\nabla u(t)}2 \le C(\norm{u_0}_{\mathcal{H}},T)<\I.
\]
This concludes the proof of Theorem \ref{thm:withpower}.
\end{proof}

\begin{proof}[Proof of Theorem \ref{thm:lowpower}]
We denote $L^p((-T,T);X)=L^p_T X$.
Our strategy for local well-posedness is to use 
the contraction argument in a complete metric space
$(\mathcal{H}^{1,2}_{T,M},d)$, where
\begin{align*}
	&\mathcal{H}^{1,2}_{T,M}:= \{f \in C((-T,T);H^1);
	\norm{f}_{\mathcal{H}^{1,2}_{T}}\le M\},\\
	&\norm{f}_{\mathcal{H}^{1,2}_{T}}:=
	\norm{f}_{L^\I_T \mathcal{H}} + \norm{f}_{L^{q_0}_T W^{1,r_0}}
	+ \norm{f{\log\Jbr{x}}}_{L^{q_0}_T L^{r_0}}
\end{align*} 
for an admissible pair $(q_0,r_0)$ with $r_0>2$, and the metric $d$ is given by
\begin{equation}\label{eq:pmetric}
	d(f,g) = \norm{f-g}_{L^\I_T L^2} + \norm{f-g}_{L^{q_0}_T L^{r_0}}.
\end{equation}
We shall show
\begin{multline*}
	Q[u](t,x) := (e^{itA} u_0)(x) \\
	+ \frac{i}{2\pi} \(\int_{0}^t e^{i(t-s)A} 
	\(\int_{\R^2} \log\frac{|\cdot-y|}{\Jbr{\cdot}} |u(s,y)|^2 dy\) u(s,\cdot) ds\)(x)\\
	-i\eta \(\int_0^t e^{i(t-s)A} (|u|^{p-1}u)(s) ds \)(x)
\end{multline*}
is a contraction map in $(\mathcal{H}^{1,2}_{T,M},d)$.
Mimicking the proof of Lemma \ref{lem:LWP}, one shows that
for any $M>0$, there exists $T>0$ such that
$Q:\mathcal{H}^{1,2}_{T,M}\to \mathcal{H}^{1,2}_{T,M}$.
To prove $Q$ is a contraction with respect to the metric $d$, 
the following estimate is crucial:
\begin{align*}
	&\Lebn{\(\int_{\R^2} \log\frac{|x-y|}{\Jbr{x}} |u_1(y)|^2 dy\) u_1
	-\(\int_{\R^2} \log\frac{|x-y|}{\Jbr{x}} |u_2(y)|^2 dy\) u_2}2\\
	&{} \le \Lebn{\(\int_{\R^2} \log\frac{|x-y|}{\Jbr{x}} |u_1(y)|^2 dy\)
	(u_1- u_2)}2\\
	&\quad + \Lebn{\(\int_{\R^2} \log\frac{|x-y|}{\Jbr{x}} (|u_1(y)|^2-|u_2(y)|^2) dy\) u_2}2\\
	&{} \le C(\Lebn{u_1}2^2 + \|\sqrt{\log\Jbr{x}}u_1\|_{L^2}^2)
	(\Lebn{u_1-u_2}2 + \Lebn{u_1-u_2}{r_0}) \\
	&\quad  + C(\Lebn{|u_1|^2 -|u_2|^2}1 + \|(|u_1|^2-|u_2|^2)\log\Jbr{x}\|_{L^1})
	(\Lebn{u_2}2 + \Lebn{u_2}{r_0})\\
	&{} \le C(\Lebn{u_1}2^2 + \|\sqrt{\log\Jbr{x}}u_1\|_{L^2}^2)
	(\Lebn{u_1-u_2}2 + \Lebn{u_1-u_2}{r_0}) \\
	&\quad  + C(\Lebn{u_1}2 +\Lebn{u_2}2 + \|u_1\log\Jbr{x}\|_{L^2}+ \|u_2\log\Jbr{x}\|_{L^2})\\
	&\quad\quad\times (\Lebn{u_2}2 + \Lebn{u_2}{r_0})\Lebn{u_1-u_2}2.
\end{align*}
By the Strichartz estimate, letting $T$ smaller if necessary, we hence obtain
\[
	d(Q[u_1],Q[u_2]) \le \frac12 d(u_1,u_2)
\]
for any $u_1,u_2 \in \mathcal{H}^{1,2}_{T,M}$.

A similar result as Lemma \ref{lem:BA} holds since
\[
		\abs{\frac{d}{dt} \Lebn{\log\Jbr{x} u(t)}2^2}
	= \abs{ \Im \int \frac{2x \log\Jbr{x}}{1+x^2}\cdot \nabla u(t) \overline{u(t)} dx}\le C \Lebn{\nabla u(t)}2.
\]
Now, we have a priori bound of $\Lebn{\nabla u(t)}2$
as in the case $2\le p<3$ of Theorem \ref{thm:withpower},
which proves the global well-posedness.
\end{proof}

\subsection*{Acknowledgments}
The author expresses his deep gratitude
to Professor Patric G\'erard for fruitful discussions. 
Deep appreciation goes to Professor Hideo Kubo 
for his valuable advice and constant encouragement.
This research is supported by Grant-in-Aid for JSPS Fellows.

\bibliographystyle{amsplain}
\bibliography{caustic}

\end{document}